\documentclass[a4paper,10pt]{extarticle}
\usepackage{authblk}
\usepackage[english]{babel}
\usepackage[utf8]{inputenc}
\usepackage[T1]{fontenc}
\usepackage{libertine}
\usepackage{url}
\usepackage{hyperref}
\usepackage{graphicx}

\usepackage{amsmath,amssymb,amsthm,amsbsy}
\usepackage{bbm} 

\usepackage{geometry}
\usepackage{fancyvrb}

\theoremstyle{plain}
\newtheorem{remark}{Remark}
\newtheorem*{remark*}{Remark}

\newtheorem{theorem}{Theorem}
\newtheorem{proposition}[theorem]{Proposition}
\newtheorem{corollary}[theorem]{Corollary}
\newtheorem{lemma}{Lemma}[section]

\newcommand{\Cc}{\mathbb{C}} 
\newcommand{\N}{\mathbb{N}} 
\newcommand{\R}{\mathbb{R}} 

\newcommand{\ba}{\mathbf{a}}

\newcommand{\bu}{\mathbf{u}}
\newcommand{\bx}{\mathbf{x}}
\newcommand{\by}{\mathbf{y}}

\newcommand{\bU}{\mathbf{U}}
\newcommand{\bX}{\mathbf{X}}
\newcommand{\bY}{\mathbf{Y}}

\newcommand{\balpha}{\boldsymbol{\alpha}}

\newcommand{\Det}[1]{\mathrm{det}\left(#1\right)}
\newcommand{\Esp}[1]{\mathbb{E}\left[#1\right]}

\newcommand{\calD}{\mathcal{D}}
\newcommand{\calU}{\mathcal{U}}
\newcommand{\calX}{\mathcal{X}}
\newcommand{\calY}{\mathcal{Y}}

\newcommand{\eqdef}{\overset{\mathrm{def}}{=}}

\title{A New Probabilistic Representation of the Alternating Zeta Function and
a New Selberg-like Integral Evaluation}

\author{Serge Iovleff}
\affil{University de Technologie de Belfort-Montb\'eliard\\
P\^ole Énergie-Informatique,\\
13 rue Ernest Thierry-Mieg,\\
90010 Belfort cedex, France

Serge.Iovleff@utbm.fr}
\affil{
Laboratoire Paul Painlev\'e, \\
CNRS U.M.R. 8524,\\
59 655 Villeneuve d'Ascq Cedex, France
}
\affil{Inria, Team Modal, Parc scientifique de la Haute-Borne\\
40, avenue Halley – Bât A – Park Plaza\\
59650 Villeneuve d’Ascq – France}
\begin{document}

\maketitle

\begin{abstract}
In this paper, we present two new representations of the alternating Zeta function.
We show that for any $s\in\Cc$ this function can be computed as a limit
of a series of determinant. We then express these determinants as the
expectation of a functional of a random vector with
Dixon-Anderson density. The generalization
of this representation to more general alternating series allows us
to evaluate a Selberg-type integral with a generalized Vandermonde determinant.
\end{abstract}

\section{Introduction}

The Dirichlet eta function is defined by the following
Dirichlet series, which converges for any complex number
having real part greater than $0$
\begin{equation}\label{eq:dirichlet}
\eta(s) = \sum_{n=1}^{\infty} \frac{(-1)^{n-1}}{n^s} = \frac{1}{1^s} - \frac{1}{2^s} + \frac{1}{3^s} - \frac{1}{4^s} + \cdots
\end{equation}

This Dirichlet series is the alternating sum corresponding
to the Dirichlet series expansion of the Riemann zeta function
$\zeta(s)$ and for this reason the Dirichlet eta function
is also known as the alternating zeta function. The following
relation holds:
$$\eta(s) = \left(1-2^{1-s}\right) \zeta(s).$$

The starting point of our work is a result from \cite{biane2001probability} (p. 456).
In their paper, they show that the sum (\ref{eq:dirichlet})
can be approximated  using an array of coefficients $(a_{n,N},\, 1\leq n\leq N)$.
Let
\begin{equation}\label{eq:anN}
a_{n,N} \eqdef \frac{1}{2} \prod_{j=1,\,j\neq n}^N \frac{  j^2}{ j^2-n^2 }
        = (-1)^{n-1}\frac{\dbinom{2N}{N-n}}{\dbinom{2N}{N}}.
\end{equation}
then
\begin{theorem}[Biane and al.]\label{th:biane}
For $\varepsilon_i,\, 1\leq i\leq N$ independent standard exponential variables, and $\Re(s)>-2N$
\begin{equation}\label{eq:expectation}
\Esp{\left( \sum_{n=1}^N \frac{\varepsilon_n}{n^2} \right)^{\frac{s}{2}} } = s \Gamma\left(\frac{s}{2}\right) \sum_{n=1}^N \frac{a_{n,N}}{n^s}
\end{equation}
where the $a_{n,N}$ are defined by (\ref{eq:anN}), and
\begin{equation}\label{eq:cvseries}
\eta_N(s) \eqdef \sum_{n=1}^N \frac{a_{n,N}}{n^s} \longrightarrow \eta(s)\mbox{ as } N\rightarrow\infty
\end{equation}
uniformly on every compact subset of $\Cc$.
\end{theorem}

\begin{remark}
The signs of the coefficients $a_{n,N}$ have been modified with respect
to those given in the previously cited article, in order to obtain convergence
in equation (\ref{eq:cvseries}) towards $\eta(s)$ rather than $-\eta(s)$.
\end{remark}

The second part of the theorem shows that, for any
$s\in\Cc$ the alternating Zeta function can be obtained by weighting
the first $N$ terms of the original series which is defined only for
$\Re(s)>0$. In this paper, we will show that the weighted finite series
$\eta_N(s)$ defined in (\ref{eq:cvseries}) can be written as
a determinant (section \ref{subsec:determinant},
proposition \ref{prop:determinant}). We will then show that this determinant
can be written as the expectation of a functional of a Dixon-Anderson random vector
(section \ref{subsec:integral}, theorem \ref{th:Dixon}).
This result is new (up to our knowledge) and seems to show that there is a
relation between the Zeta function and the theory of random matrices.

In section \ref{sec:AltenatingSeries}
we give a generalization of the representations given in section
\ref{sec:Representations} for general series and, by computing the expectation of
these representations, we obtain the evaluation of two Selberg integrals
involving a generalized Vandermonde determinant
(theorem \ref{th:Averaged-Gen-Vandermonde}).
\section{New Representations of the Alternating Zeta Function}
\label{sec:Representations}

\subsection{Determinant Representation}
\label{subsec:determinant}

We start with the following result
\begin{proposition}\label{prop:determinant}
With the notations and conditions given in theorem (\ref{th:biane}) we have
$$
\eta_N(s) =
\frac{1}{2}
\begin{vmatrix}
1       & \frac{1}{3!}    & \ldots & \frac{1}{(2N-1)!} \\
2^{1-s} & \frac{ 2^3}{3!} & \ldots & \frac{2^{2N-1}}{(2N-1)!}  \\
\vdots  & \vdots          &        & \vdots \\
N^{1-s} & \frac{N^3}{3!}  & \ldots &  \frac{N^{2N-1}}{(2N-1)!}
\end{vmatrix}.
$$
\end{proposition}
\begin{proof}
Let $a_n = n^{2}$, $Q_N(x) = \prod_{n=1}^N \left(1-x/a_n \right)$,
and $P_N$ a polynomial of degree $N-1$ over $\Cc[X]$, with the
convention $ P_N(x) = \sum_{n=1}^N c_{n,N} x^{n-1}$. An adaptation
of the arguments given in annex \ref{sec:pfd} shows that 
$$
\frac{P_N(x)}{Q_N(x)} = \sum_{n=1}^N P_N\left({a_n}\right)
\prod_{ \substack{j=1\\ j\neq n} }^N
\left(\frac{1}{\left(1-\frac{a_n}{a_j}\right)} \frac{1}{1-\frac{x}{a_j}}\right).
$$
Setting $x=0$, we have $Q_N(0)=1$ and we get
$$
P_N(0) = 2 \sum_{n=1}^N a_{n,N} P_N\left(n^2\right).
$$
We choose $P_N$ as the polynomial of degree $N-1$ such that
$P_N(a_n) = n^{-s}$ for $n=1,\ldots N$.
The coefficients $(c_{n,N})_{n=1}^N$  of $P_N$ are solutions of the
Vandermonde system
\begin{equation}
\label{eq:system}
\begin{pmatrix}
1      & 1      & \ldots & 1 \\
1      & 2^2    & \ldots & 2^{2(N-1)} \\
\vdots & \vdots &        & \vdots \\
1      & N^2    & \ldots & N^{2(N-1)}
\end{pmatrix}
\begin{pmatrix}
c_{1,N}\\ c_{2,N}\\ \vdots \\ c_{N,N}
\end{pmatrix}
=
\begin{pmatrix}
1\\ 2^{-s}\\ \vdots \\ N^{-s}
\end{pmatrix}.
\end{equation}
The Vandermonde matrix is invertible (see annex \ref{app:inverse})
showing that the polynomial $P_N$ is uniquely determined.
We are only interested in $c_{1,N} = P_N(0)$. Let us denote $V_N$ the Vandermonde matrix
in (\ref{eq:system}) and $V^{(s)}_N$ the matrix obtained by replacing the first column
of $V_N$ by the right-end term.

Using Cramer's rule, we get
\begin{equation}\label{eq:c1N}
c_{1,N} =\frac{
\begin{vmatrix}
1      & 1      & \ldots & 1 \\
2^{-s} & 2^2    & \ldots & 2^{2(N-1)} \\
\vdots & \vdots &        & \vdots \\
N^{-s} & N^2    & \ldots & N^{2(N-1)}
\end{vmatrix}
}
{\Det{V_N}} = \frac{\Det{V^{(s)}_N}}{\Det{V_N}}.
\end{equation}
The Vandermonde Determinant of $V_N$ is
\begin{equation}\label{eq:VN}
 \Det{V_N}
 = \prod_{1 \mathop \le i \mathop < j \mathop \le N} \left({j^2 - i^2}\right)
 = \prod_{1 \mathop \le i \mathop < j \mathop \le N} \left({j - i}\right)\left({j + i}\right)
 = \prod_{j=2}^{N} (j-1)! \frac{(2j-1)!}{j!}
 = \frac{1}{N!} \prod_{n=1}^{N-1} (2n+1)!
\end{equation}
and thus, we have
\begin{equation*}\label{eq:c1Nbis}
c_{1,N} = 
\begin{vmatrix}

1       & \frac{1}{3!}    & \ldots & \frac{1}{(2N-1)!} \\
\frac{2}{2^s} & \frac{ 2^3}{3!} & \ldots & \frac{2^{2N-1}}{(2N-1)!}  \\
\vdots  & \vdots          &        & \vdots \\
\frac{N}{N^s} & \frac{N^3}{3!}  & \ldots &  \frac{N^{2N-1}}{(2N-1)!}
\end{vmatrix}.
\end{equation*}
ending the proof.
\end{proof}

\begin{remark}
Observe that using the expression (\ref{eq:c1N}),
it is immediate that $\eta_N(s) = \frac{1}{2}$ if $s=0$ and
$\eta_N(s) =0$ if $s=-2,-4,\ldots,-2(N-1)$. 
\end{remark}
\begin{remark}
An alternative (and more direct) proof could have be used using the
representation of the generalized Vandermonde determinant given in lemma
\ref{lemma:SumGenVandermonde}.
\end{remark}

\subsection{Probabilistic Representation}
\label{subsec:integral}

We observe that $\Det{V^{(s)}_N}$
(equation \ref{eq:c1N}) is a generalized Vandermonde determinant.
Using the argument used in (\cite{ZHANG2008300}),
we get the following lemma
\begin{lemma}
\label{lemma:Gen-Vandermonde}
For $N>1$, let 
$$
{\Gamma_{N-1}\left( \frac{s}{2} \right)}
\eqdef\frac{2}{s}
\prod_{n=1}^{N-1}
\frac{\left(1+{\frac {1}{n}}\right)^{s/2}}{1+{\frac {s}{2n}}},\quad s\neq 0,-2,-4,\ldots
$$
then for all $s\in\Cc$, $s\neq -2, -4,\ldots, -2(N-1)$ we have
\begin{equation}\label{eq:DetVsN}
\Det{V^{(s)}_N}= 
\frac{2(N-1)! }{s{\Gamma_{N-1}\left( \frac{s}{2} \right)}}
\int_{1}^{2^2}
dx_1
\int_{2^2}^{{3^2}}
dx_2
\ldots
\int_{(N-1)^2}^{{N^2}}
dx_{N-1}
\prod_{1\leq i<j\leq N-1} (x_{j}-x_{i})
\prod_{n=1}^{N-1} \left(\frac{x_n}{n(n+1)}\right)^{s/2}.
\end{equation}
\end{lemma}
\begin{proof}
Observe first that $\left(s\Gamma_{N-1}\left(\frac{s}{2}\right) \right)^{-1}$
is defined for all $N>1$ and all $s\in\Cc$. Next, we have
\begin{eqnarray*}
\Det{V^{(s)}_N} & = &
\begin{vmatrix}
1      & 1      & \ldots & 1 \\
2^{-s} & 2^2    & \ldots & 2^{2(N-1)} \\
3^{-s} & 3^2    & \ldots & 3^{2(N-1)} \\
\vdots & \vdots &        & \vdots \\
N^{-s} & N^2    & \ldots & N^{2(N-1)}
\end{vmatrix}
 =  \prod_{n=2}^N \frac{1}{n^s}\times
\begin{vmatrix}
1 & 1           & \ldots & 1 \\
1 & 2^{2(1+s/2)}& \ldots & 2^{2(N-1+s/2)} \\
1 & 3^{2(1+s/2)}& \ldots & 3^{2(N-1+s/2)} \\
\vdots          & \vdots &        & \vdots \\
1 & N^{2(1+s/2)}& \ldots & N^{2(N-1+s/2)}
\end{vmatrix}\nonumber\\
& = & \prod_{n=2}^N \frac{1}{n^s}\times
\begin{vmatrix}
 2^{2(1+s/2)}-1                & \ldots & 2^{2(N-1+s/2)}-1 \\
 3^{2(1+s/2)}- 2^{2(1+s/2)}    & \ldots & 3^{2(N-1+s/2)}-2^{2(N-1+s/2)} \\
\vdots                         & \vdots & \vdots \\
 N^{2(1+s/2)}-(N-1)^{2(1+s/2)} & \ldots & N^{2(N-1+s/2)}-(N-1)^{2(N-1+s/2)}
\end{vmatrix}\nonumber\\
& = &
\prod_{n=1}^{N-1}\frac{n+s/2}{(n+1)^{s}}
\times
\begin{vmatrix}
 \int_{1}^{2^2} x_1^{s/2}dx_1  & \ldots & \int_{1}^{2^2} x_1^{N-2+s/2}dx_1 \\
 \int_{2^2}^{3^2} x_2^{s/2}dx_2& \ldots & \int_{2^2}^{3^2} x_2^{N-2+s/2}dx_2 \\
\vdots                         & \vdots & \vdots \\
 \int_{(N-1)^2}^{N^2} x_{N-1}^{s/2}dx_{N-1}  & \ldots & \int_{(N-1)^2}^{N^2} x_{N-1}^{N-2+s/2}dx_{N-1} \\
\end{vmatrix}\nonumber\\
& = &
\prod_{n=1}^{N-1}\frac{n+s/2}{(n+1)^{s}}
\int_{1}^{2^2}dx_1
\int_{2^2}^{{3^2}}dx_2
\ldots\int_{(N-1)^2}^{N^2}dx_{N-1}
\begin{vmatrix}
 x_1^{s/2} & \ldots & x_1^{N-2+s/2} \\
\vdots     & \vdots & \vdots \\
 x_{N-1}^{s/2} & \ldots &  x_{N-1}^{N-2+s/2} \\
\end{vmatrix}
\nonumber\\
& = &
\prod_{n=1}^{N-1}\frac{n+s/2}{(n+1)^{s}}
\int_{1}^{2^2}dx_{1}\int_{2^2}^{{3^2}}dx_{2}
\ldots\int_{(N-1)^2}^{N^2}dx_{N-1}
\begin{vmatrix}
 1     & x_1   &\ldots & x_1^{N-2} \\
\vdots &\vdots &\vdots & \vdots \\
 1     & x_{N-1}   &\ldots &  x_{N-1}^{N-2} \\
\end{vmatrix}
\prod_{n=1}^{N-1} x_n^{s/2} \nonumber\\
& = &
\frac{2(N-1)!}{s{\Gamma_{N-1}\left( \frac{s}{2} \right)}}
\int_{1}^{2^2}
dx_1
\int_{2^2}^{{3^2}}
dx_2
\ldots
\int_{(N-1)^2}^{{N^2}}
dx_{N-1}
\prod_{1\leq i<j\leq N-1} (x_{j}-x_{i})
\prod_{n=1}^{N-1} \left(\frac{x_n}{n(n+1)}\right)^{s/2}.
\end{eqnarray*}
\end{proof}
Let $\mathcal{D}_{N-1}(\bx;\balpha,\ba)$, with $\balpha=(\alpha_1,\alpha_2,\ldots,\alpha_N)$,
$\alpha_n>0$ and $\ba=(a_1,a_2,\ldots,a_N)$ with
$a_1<a_2<\ldots<a_N$, denote the Dixon-Anderson
probability density function (pdf) over the domain
$\calX_{N-1} = \left\{a_1<x_1<a_2<\ldots<x_{N-1}<a_N\right\}$ (see \cite{forrester2010log}
page 138)
\begin{equation}\label{eq:Dixon-Gen}
\mathcal{D}_{N-1}(\bx;\balpha,\ba) =
\frac{\Gamma\left(\sum_{n=1}^{N}\alpha_n\right)}{\prod_{n=1}^{N}\Gamma(\alpha_n)}
\frac{\prod_{1\leq i<j\leq N-1} (x_j-x_i)}
     {\prod_{1\leq i<j\leq N} (a_j-a_i)^{\alpha_j+\alpha_i-1}}
\prod_{n=1}^{N-1}\prod_{i=1}^{N} |x_n-a_i|^{\alpha_i-1}. 
\end{equation}
Taking $\balpha=\mathbf{1}_{N}$ (i.e. $\alpha_1=\alpha_2=\ldots=\alpha_N=1$)
and $a_n=n^2,\; n=1,\ldots,N$, we get that
\begin{equation}\label{eq:Dixon}
\mathcal{D}_{N-1}(\bx;\balpha,\ba) = (N-1)! \frac{\prod_{1\leq i<j\leq N-1} (x_j - x_i)}
                                             {\prod_{1\leq i<j\leq N} (j^2 - i^2)}
\end{equation}
is a pdf over $\calX_{N-1}=\left\{1<x_1<2^2<\ldots<x_{N-1}<N^2\right\}$.
From the previous lemma, we obtain the following theorem
\begin{theorem}\label{th:Dixon}
Let $\bX=(X_1,\ldots,X_{N-1}) $ be a random vector with Dixon-Anderson
distribution given by (\ref{eq:Dixon}), then
\begin{equation}\label{eq:dixon-esperance}
 \frac{1}{s}\frac{1}{\Gamma_{N-1}\left(\frac{s}{2}\right)}
 \Esp{\prod_{n=1}^{N-1} \left(\frac{X_n}{n(n+1)}\right)^{s/2}}
  \longrightarrow \eta(s)\mbox{ as } N\rightarrow\infty
\end{equation}
uniformly on every compact subset of $\Cc$.
\end{theorem}

\begin{proof}
Using the expression of $\Det{V_N}$ given in equation (\ref{eq:VN})
we find the expectation given in (\ref{eq:dixon-esperance}) for $N$ fixed.

The Gamma function can be defined as an infinite
product for all complex numbers $z$ except the non-positive
integers
$$\Gamma (z)=
\frac {1}{z}\prod _{n=1}^{\infty }
\frac {\left(1+{\frac {1}{n}}\right)^{z}}{1+{\frac {z}{n}}}
$$
and for any $s\in\Cc$ (the case $s=0$ is handled by continuity)
$$
\frac{2}{s\Gamma_{N-1}\left( \frac{s}{2} \right)}
\xrightarrow[N\rightarrow\infty]{}
\frac{1}{\Gamma\left( 1+\frac{s}{2} \right)}
$$
uniformly on every compact of $\Cc$.
For any $s\in \Cc$, we have
$$
\eta_N(s) =\frac{1}{2} \frac{\Det{V^{(s)}_N}}{\Det{V_N}}
= \frac{1}{s}\frac{1}{\Gamma_{N-1}\left(\frac{s}{2}\right)}
\Esp{\prod_{n=1}^{N-1} \left(\frac{X_n}{n(n+1)}\right)^{s/2}}
$$
and thus the conclusion of theorem \ref{th:biane} occurs as well.
\end{proof}

\subsection{A Result Related to Theorem \ref{th:Dixon}}

One interesting fact about the Dixon-Anderson distribution given in
(\ref{eq:Dixon}) is that it is invariant under some linear transformations.
More precisely if $\bX$ is a Dixon-Anderson random vector with pdf
$\mathcal{D}_{N}(\bx;\balpha,\ba)$ and
$(u,v)\in \R^\star\times\R$ then $\bY=u\bX+v\mathbf{1}_N$ is a Dixon-Anderson random
vector with pdf $\mathcal{D}_{N}(\bx;\balpha,u\ba+v\mathbf{1}_N)$.
Using this property we can renormalize the random vector $\bX$ over $[0,1]$ 
by using the change of variable $\bY=(\bX-\mathbf{1}_{N-1})/(N^2-1)$ giving
us the identity
$$
\eta_N(s) = \frac{1}{s\Gamma_{N-1}\left( \frac{s}{2} \right)}
\left(\frac{(N^2-1)^{(N-1)}}{N!(N-1)!}\right)^{\frac{s}{2}}
\Esp{\prod_{n=1}^{N-1} \left({Y_n+\frac{1}{N^2-1}}\right)^{\frac{s}{2}} }.
$$
We have the following theorem
\begin{theorem}
Let $\psi_N(x;s)$ denote the application
$$
\psi_N(x;s) =\frac{2}{\Gamma\left(1+\frac{s}{2} \right)} \left(\frac{(N^2-1)^{(N-1)}}{N!(N-1)!}\right)^{\frac{s}{2}}
\Esp{\prod_{n=1}^{N-1} \left|{Y_n-\frac{x^2-1}{N^2-1}}\right|^{\frac{s}{2}} }.
$$
Then for all $n\in\N^\star$
$$
\psi_N\left({n;s} \right)
\xrightarrow[N\rightarrow\infty]{}\frac{1}{n^s}
$$
and for $n=0$
$$
\psi_N\left(0;s \right)
\xrightarrow[N\rightarrow\infty]{}
\eta(s).
$$
\end{theorem}
\begin{proof}
The case $n=0$ is a consequences of the theorem \ref{th:Dixon}.
Let $b_{n,N} = \frac{n^2-1}{N^2-1}$ for $n=1,\ldots,N$ and
$\calY_{N-1} = \left\{0=b_{1,N}<y_1<b_{2,N}<y_2<\ldots<y_{N-1}<b_{N,N}=1 \right\}$.
Taking $x=n$ we can compute the value of $\psi_N(x;s)$. We have
\begin{equation*}
\Esp{\prod_{n=1}^{N-1} \left|{Y_n-b_{n,N}}\right|^{\frac{s}{2}} }
= \frac{(N-1)!}{\displaystyle\prod_{1\leq i<j\leq N} \left(b_{j,N} - b_{i,N}\right)}
\int_{\calY_{N-1}}d\by
\prod_{1\leq i<j\leq N-1} (y_{j}-y_{i}) \prod_{k=1}^{N-1} |y_k - b_{n,N}|^{s/2}.
\end{equation*}
The integral of the right hand side is (Consider the pdf given in (\ref{eq:Dixon-Gen})
with $\alpha_k=1$ if $k\neq n$ and $\alpha_k=1+s/2$ otherwise)
\begin{multline*}
\int_{\calY_{N-1}}
d\by\prod_{1\leq i<j\leq N-1} (y_{j}-y_{i}) \prod_{k=1}^{N-1} (y_k - b_{n,N})^{s/2}
\\=
\frac{\Gamma\left(1+\frac{s}{2}\right)}{\Gamma\left(N+\frac{s}{2}\right)}
\displaystyle\prod_{1\leq i<j\leq N} \left(b_{j,N} - b_{i,N}\right)
\prod_{k=1}^{n-1} \left(b_{n,N} - b_{k,N}\right)^{s/2}
\prod_{k=n+1}^{N} \left(b_{k,N} - b_{n,N}\right)^{s/2}
\end{multline*}
From this we deduce that when $n\neq 1$ we have
\begin{eqnarray*}
\psi_N(n;s) &=& \left(\frac{\prod_{k=1}^{n-1} \left(n^2 - k^2\right)
\prod_{k=n+1}^{N} \left(k^2 - n^2\right)}{N!(N-1)!}\right)^{{s}/{2}}
\frac{\Gamma(N)}{\Gamma\left(N+\frac{s}{2}\right)}\\
& = &
\left(\frac{(N-n)!(N+n)!}{n^2\,N!(N-1)!}\right)^{{s}/{2}}
\frac{\Gamma(N)}{\Gamma\left(N+\frac{s}{2}\right)}
\\
&=& \left(\frac{1}{|a_{n,N}|n^2}\right)^{s/2}\,N^{s/2}\,
\frac{\Gamma(N)}
     {\Gamma\left(N+\frac{s}{2}\right)}
\end{eqnarray*}
with $a_{n,N}$ defined in (\ref{eq:anN}).
In the case $n=1$ we find directly
$$
\psi_N(1;s) = N^{\frac{s}{2}}
 \frac{\Gamma(N)}
      {\Gamma\left(N+\frac{s}{2}\right)}
$$
Taking the limit and observing that
$\Gamma\left(N+\frac{s}{2}\right)\sim \Gamma\left(N\right)N^{s/2}$
as $N\rightarrow+\infty$ end the proof.
\end{proof}

\section{Averaged Alternating Random Series}
\label{sec:AltenatingSeries}

\subsection{A generalization of proposition \ref{prop:determinant}
and theorem \ref{th:Dixon}}
let $s\in\Cc$ with $s\neq -2,-4,\ldots,-2(N-1)$
and let $u_1<u_2<\ldots<u_N$ be an increasing sequence of real numbers
in $\R^\star$.
From this sequence, we define the $N\times N$ generalized Vandermonde determinant
\begin{equation*}
V_N^{(s/2)}(\bu) =
\begin{vmatrix}
  u_1^{-s/2} & u_1    & u_1^2  & \cdots & u_1^{N-1} \\
  u_2^{-s/2} & u_2    & u_2^2  & \cdots & u_2^{N-1} \\
\vdots       & \vdots & \ddots & \vdots & \vdots \\
  u_N^{-s/2} & u_N    & u_N^2  & \cdots & u_N^{N-1} \\
\end{vmatrix}
\end{equation*}
with $\bu$ denoting the ordered vector $(u_1,\ldots,u_N)$\footnote{The
reader will be aware that in section \ref{sec:Representations},
$V_N^{(s/2)}$ represented a \textbf{matrix}, whereas from now the
notation $V_N^{(s/2)}(\bu)$ represents a \textbf{determinant}}.
\begin{lemma}
\label{lemma:SumGenVandermonde}
let $u_1<u_2<\ldots<u_N$ be an increasing sequence of real numbers
in $\R^\star$. The following hold 
\begin{equation}\label{eq:Alternating-Sum}
\frac{V_N^{(s/2)}(\bu)}{V_N^{(0)}(\bu)}
= \sum_{n=1}^N (-1)^{n-1} \frac{1}{u_n^{s/2}}
\prod_{\substack {1\le j\le N\\j\ne n}} \frac{u_j}{|u_j-u_n|}.
\end{equation}
\end{lemma}
\begin{proof}
Observe that $V_N^{(0)}(\bu)$ denotes the usual determinant
of a Vandermonde matrix.
Let us denote by $V_N^{-n}$ the following determinant
$$
V_N^{-n} =
\begin{vmatrix}
 u_1     & u_1^2     & \cdots & u_1^{N-1} \\
 u_2     & u_2^2     & \cdots & u_2^{N-1} \\
\vdots   & \vdots    & \vdots & \vdots \\
 u_{n-1} & u_{n-1}^2 & \cdots & u_{n-1}^{N-1} \\
 u_{n+1} & u_{n+1}^2 & \cdots & u_{n+1}^{N-1} \\
\vdots   & \vdots    & \vdots & \vdots \\
 u_N     & u_N^2     & \cdots & u_N^{N-1} \\
\end{vmatrix},\qquad n=1,\ldots,N.
$$
Then, it is obvious that
$$
V_N^{-n} =
\left(\prod_{\substack {1\le j\le N\\j\ne n}} u_j\right)
V_{N-1}^{(0)}(u_1,\ldots,u_{n-1},u_{n+1},\ldots,u_N).
$$
By looking closely at the missing products, we obtain that
$$
V_N^{-n} = 
\left(\prod_{\substack {1\le j\le N\\j\ne n}} u_j\right)
\frac{V_N^{(0)}(\bu)}
     {\prod_{j=1}^{n-1} (u_n-u_j)
     \prod_{l=n+1}^{N} (u_{l}-u_{n})}
= (-1)^{n-1} V_N^{(0)}(\bu)
 \prod_{\substack {1\le j\le N \\j \ne n}} \frac{u_j}{u_j-u_n}.
$$ 
We have thus
\begin{equation*}
\label{eq:GenVandermondeDet}
\frac{V_N^{(s/2)}(\bu)}{ V_N^{(0)}(\bu) }
= \sum_{n=1}^N (-1)^{n+1} \frac{1}{u_n^{s/2}}
\frac{V_N^{-n}}{ V_N^{(0)}(\bu)} 
= \sum_{n=1}^N \frac{1}{u_n^{s/2}}
               \prod_{\substack {1\le j\le N\\j\ne n}}
                \frac {u_j}{u_j - u_n}.
\end{equation*}
As the sequence $(u_n,\, n=1,\ldots N)$ is strictly increasing,
the sum (\ref{eq:Alternating-Sum}) is alternating as announced.
\end{proof}
We have the following result which generalize theorem
\ref{th:Dixon}
\begin{proposition}
\label{prop:alternating-Sum}
Let $0<u_1<u_2<\ldots<u_N$ be an arbitrary increasing sequence of
positive real number, let $s\in\Cc$ with
$s\neq -2,-4,\ldots,-2(N-1)$ and let 
$\bX=(X_1,\ldots,X_{N-1})$ denote a random vector
with Dixon-Anderson density  $\mathcal{D}_{N-1}(\bx;\mathbf{1}_{N-1},\bu)$ then
\begin{equation}
\label{eq:alternating-Sum}
\frac{2}{s\Gamma_{N-1}\left(\frac{s}{2}\right)}
{N^{s/2}}\Esp{ \frac{\prod_{n=1}^{N-1} X_n^{s/2}}{\prod_{n=1}^{N} u_n^{s/2}}}
=\frac{V_N^{(s/2)}(\bu)}{ V_N^{(0)}(\bu) }=\sum_{n=1}^N  (-1)^{n-1}\frac{1}{u_n^{s/2}}
\prod_{\substack{1 \le j\le N\\j\ne n}}\frac{u_j}{|u_j-u_n| }
\end{equation}
\end{proposition}
The proof follows the same steps as in lemma
\ref{lemma:Gen-Vandermonde} and is left to the reader.
The reader can also note that taking $u_n=n^2$ we get
the expression obtained in section \ref{subsec:integral}.

\begin{remark}
It is clear that if the left hand side of the equation (\ref{eq:alternating-Sum})
converges in some sense as $N\rightarrow\infty$ to a well defined function in $s$,
then this function will be equal to $1$ when $s=0$ and equal to $0$ when
$s=-2k$, $k\in\N^\star$. 
\end{remark}

Finally we note that a similar lemma have been proved in \cite{biane2001probability}
using exponential random variables
\begin{lemma}
For $(\varepsilon_n,1\leq n\leq N)$ independent standard exponential variables,
and $u_1,u_2,\ldots,u_N$ an arbitrary sequence of numbers all distincts and strictly
positive, if $\Re(s/2)>-N$
$$
\Esp{\left(\sum_{n=1}^N \frac{\varepsilon_n}{u_n} \right)^{s/2}} = \Gamma\left(1+\frac{s}{2} \right)
\sum_{n=1}^N \left(\frac{1}{u_n} \right)^{s/2}
\prod_{\substack{1 \le j\le N\\j\ne n}}\frac{u_j}{u_j-u_n}.
$$
\end{lemma}

\subsection{A family of joint density probability}
Let $\bu\in \R^N$ and $\bx\in\R^{N-1}$ be two interlacing vectors
in the sense that they lie in the region $\calX_N^\prime$ defined as
$$
\calX_N^\prime=\left\{0<u_1<x_1<u_2<\ldots < u_{N-1} < x_{N-1} < u_N\right\}.
$$
Let $g$ denote a positive function over $\R_+$ to be precised hereafter.
We define a joint density probability over $\calX_N^\prime$ by putting
\begin{eqnarray}\label{eq:pdf-jointe}
f_{\bX,\bU}(\bx,\bu)& =& \frac{(N-1)!N!}{Z_N}
{V_{N-1}^{(0)}(\bx)}\;\ V_N^{(0)}(\bu)
\prod_{n=1}^N g(u_n) \nonumber\\
&=&
\mathcal{D}_{N-1}(\bx;\mathbf{1}_{N-1},\bu)\;\frac{N!}{Z_N}
\left(V_N^{(0)}(\bu)\right)^2
\prod_{n=1}^N g(u_n).
\end{eqnarray}
It is quite evident that if $(\bX,\bU)$ are two random interlacing
vectors with such distribution, then the distribution
of $\bX$ conditional to $\bU=\bu$ is a Dixon-Anderson random
vector of density $\calD_{N-1}(\bx;\mathbf{1}_{N-1}, \bu)$.
The marginal distributions of $\bX$ and $\bU$ are obtained
by integrating the probability density function (\ref{eq:pdf-jointe})
with respect to $\bu$ and $\bx$ respectively.
Integrating with respect to $\bx$, we find that the density of $\bU$ is
\begin{equation}\label{eq:pdf-u}
f_{\bU}(\bu)= \frac{N!}{Z_N}  \left(V_N^{(0)}(\bu)\right)^2
\prod_{n=1}^N g(u_n).
\end{equation}
over the domain $\calU_N=\left\{0<u_1<u_2<\ldots < u_N\right\}$.
Note that, as $f_{\bU}$ is invariant under permutation, we have
$$
Z_n = \int_0^{\infty} du_1\ldots \int_0^{\infty} du_n
\left(V_N^{(0)}(\bu)\right)^2
\prod_{n=1}^N g(u_n)
$$
assuming the integral exists. Thus if $(\bX,\bU)$ are random vectors
with joint probability density function (\ref{eq:pdf-jointe})
and $s\neq -2, -4, \ldots, -2(N-1)$ then it follows from
identity given in (\ref{eq:alternating-Sum}) that
\begin{equation}\label{eq:equality}
\frac{2}{s\Gamma_{N-1}\left(\frac{s}{2}\right)}
N^{s/2}\Esp{ \frac{\prod_{n=1}^{N-1} X_n^{s/2}}{ \prod_{n=1}^{N}{U_n^{s/2}}} }
=\Esp{
\frac{V_N^{(s/2)}(\bU)}{V_N^{(0)}(\bU)}
}
\end{equation}
assuming again that the expectations involved in this equality exist and are finite.

There is two obvious choices for $g$ allowing us
to compute theses expectations: the Jacobi ensemble and the Laguerre ensemble.

\subsubsection{The Jacobi Ensemble}
We set $g(u)= u^{a-1} (1-u)^{b-1} \mathbb{1}_{(0,1)}(u)$ with $a,b>0.$
In this case, the distribution of $\bU$ conditional to
$\bX=\bx$ is a Dixon-Anderson random
vector of density $D_{N}\left(\bu;(a,\mathbf{1}_{N-1},b), (0,\bx,1)\right)$
and the marginal distribution of $\bU$ is a Selberg density $S_N(\bu;a,b,1)$
(see \cite{forrester2012fuchsian}) with  $S_N(\bu;a,b,\lambda)$ given by
\begin{equation}\label{eq-Selberg-Density}
S_N(\bu;a,b,\lambda)=\frac{N!}{S_N(a,b,\lambda)}\left(V_N^{(0)}(\bu)\right)^{2\lambda}
\prod_{n=1}^N u_n^{a-1} (1-u_n)^{b-1} 
\end{equation}
when supported on  $\calU_N=\left\{ 0<u_1<u_2<\ldots < u_N<1\right\}$.
$S_N(a,b,\lambda)$ denotes the Selberg's integral formula
(see \cite{andrews_askey_roy_1999}, chapitre 8. We choose
the definition given in this reference rather than the one given
in \cite{forrester2012fuchsian}). We have thus
$$
Z_N = S_N(a,b,1) =
\prod_{n=0}^{N-1} \frac{\Gamma(a+n)\Gamma(b+n)\Gamma(2+n)}
{\Gamma(a+b-1+N+n)}
$$
Integrating (\ref{eq:pdf-jointe}) with respect to $\bu$ 
gives the marginal density of $\bX$
\begin{eqnarray}\label{eq:pdf-jacobi-x}
f_{\bX}(\bx; a, b)&=&\frac{(N-1)!N!}{ S_N(a,b,1)} \frac{\Gamma(a)\Gamma(b)}
                         {\Gamma(a+b-1+N)}
 \left(V^{(0)}_{N-1}(\bx)\right)^2
\prod_{n=1}^{N-1} x_n^{a} (1-x_n)^{b}\nonumber\\
&=& \frac{(N-1)!}{S_{N-1}(a+1,b+1,1)} \left(V^{(0)}_{N-1}(\bx)\right)^2
\prod_{n=1}^{N-1} x_n^{a} (1-x_n)^{b}
\end{eqnarray}
i.e. the marginal density of $\bX$ is the Selberg density $S_{N-1}(\bx;a+1,b+1,1)$
supported over
$$\calX_{N-1}=\left\{ 0<x_1<x_2<\ldots <x_{N-1} <1\right\}.$$

\subsubsection{The Laguerre Ensemble}
We set now $g(u)= u^{a-1} e^{-u/b}\mathbb{1}_{(0,+\infty)}(u)$ with $a,b>0$.
The joint density of $(\bX,\bU)$ can be obtained as a limit of the Jacobi ensemble
case by changing variables $u_n = v_n/L$,
$x_n = y_n/L$, replacing $b-1$ by $L/\theta$ and by taking the limit
$L\rightarrow\infty$. We have in this case
$$
Z_N = W_N(a,\theta) = \lim_{L\rightarrow\infty}
\frac{S_N(a,L/\theta+1,1)}{L^{(a+N)N}}
= \theta^{(a+N)N} \prod_{n=0}^{N-1}\Gamma(a+n)\Gamma(2+n).
$$
The marginal distribution of $\bU$ is a Laguerre density
$$
L(\bu;a,\theta) = \frac{N!}{W_N(a,\theta)}  \left(V^{(0)}_{N}(\bu)\right)^2
\prod_{n=1}^{N} u_n^{a} e^{-u_n/\theta}
$$
supported on $\calU_N=\left\{ 0<u_1<u_2<\ldots < u_N<1\right\}$.
Integrating (\ref{eq:pdf-jointe}) with respect to $\bu$ 
gives the marginal density of $\bX$
\begin{eqnarray}\label{eq:pdf-laguerre-x}
f_{\bX}(\bx; a, \theta)&=&\frac{(N-1)!N!}{ W_N(a,\theta)} \theta^{a+N}\Gamma(a)
 \left(V^{(0)}_{N-1}(\bx)\right)^2\prod_{n=1}^{N-1} x_n^{a} e^{-x_n/\theta}\nonumber\\
&=& \frac{(N-1)!}{W_{N-1}(a+1,\theta)} \left(V^{(0)}_{N-1}(\bx)\right)^2
\prod_{n=1}^{N-1} x_n^{a} e^{-x_n/\theta}
\end{eqnarray}
i.e. the marginal density of $\bX$ is a Laguerre density $L_{N-1}(\bx;a+1,\theta)$.

\subsection{Main Result}

\begin{theorem}
\label{th:Averaged-Gen-Vandermonde}
Let $a,b,\theta>0$ and $\Re(a-s/2)>0$ and let $\bU$ be a random vector
of $\R^N$.
If the distribution of $\bU$ is the Selberg density $S_N(\bu;a,b,1)$ supported
on  $\calU_N=\left\{ 0<u_1<u_2<\ldots < u_N<1\right\}$ then
$$
\Esp{\frac{V_N^{(s/2)}(\bU)}{ V_N^{(0)}(\bU)}}
= \frac{2N^{s/2}}{s\Gamma_{N-1}(s)}
\frac{\Gamma\left(a+b-1+N\right)}
     {\Gamma\left(a-\frac{s}{2}+b-1+N\right)}
\frac{\Gamma\left(a-\frac{s}{2}\right)}
     {\Gamma(a)}.
$$
If the distribution of $\bU$ is the Laguerre ensemble density
$L(\bu;a,\theta)$  supported on $\calU_N^\prime=\left\{ 0<u_1<u_2<\ldots < u_N\right\}$
then
$$
\Esp{\frac{V_N^{(s/2)}(\bU)}{ V_N^{(0)}(\bU)}}
= \frac{2}{s\Gamma_{N-1}(s)} \left(\frac{N}{\theta} \right)^{s/2}
\frac{\Gamma\left(a-\frac{s}{2}\right)}
     {\Gamma(a)}.
$$
\end{theorem}
\begin{proof}
We have
\begin{eqnarray*}
\Esp{\frac{V_N^{(s/2)}(\bU)}{ V_N^{(0)}(\bU)}} &=& \frac{2}{s\Gamma_{N-1}\left(\frac{s}{2}\right)} N^{s/2}
\Esp{ \frac{\prod_{n=1}^{N-1} X_n^{s/2}}{ \prod_{n=1}^{N}{U_n^{s/2}}} }\\
& =&\frac{2N^{s/2}}{s\Gamma_{N-1}\left(\frac{s}{2}\right)}\frac{N!(N-1)!}{Z_N}
\int_{\calX^\prime_N}
\frac{\prod_{n=1}^{N-1} x_n^{s/2}}{ \prod_{n=1}^{N}{u_n^{s/2}}}
\;V_{N-1}^{(0)}(\bx)\;
V_N^{(0)}(\bu)
\prod_{n=1}^N g(u_n) d\bu d\bx.
\end{eqnarray*}
We integrate with respect to $\bu$. In the Jacobi ensemble case, we get
\begin{eqnarray*}
\Esp{\frac{V_N^{(s/2)}(\bU)}{ V_N^{(0)}(\bU)}}&=&
\frac{2N^{s/2}}{s\Gamma_{N-1}\left(\frac{s}{2}\right)}\frac{N!(N-1)!}{S_N(a,b,1)}
\frac{\Gamma\left(a-\frac{s}{2}\right)\Gamma(b)}{\Gamma\left(a-\frac{s}{2}+b-1+N\right)}
\int_{\calX_N}
\left(V_{N-1}^{(0)}(\bx)\right)^2 \prod_{n=1}^N x_n^a (1-x_n)^b d\bx\\
&=&
\frac{2N^{s/2}}{s\Gamma_{N-1}\left(\frac{s}{2}\right)}\frac{N!}{S_N(a,b,1)}
\frac{\Gamma\left(a-\frac{s}{2}\right)\Gamma(b)}{\Gamma\left(a-\frac{s}{2}+b-1+N\right)}
S_{N-1}(a+1,b+1,1)
\end{eqnarray*}
giving after some elementary simplifications the announced result.
The Laguerre ensemble case can be obtain either by integration, or by replacing
$b-1$ by $L/\theta$ and by taking the limit $L\rightarrow\infty$. We let the details
to the reader.
\end{proof}
Finally, we have the following corollary
\begin{corollary}
Let $a,b,\theta>0$ and $\Re(a-s/2)>0$. Then
\begin{multline*}
\int_0^1 du_1\ldots\int_0^1 du_n
V_N^{(s/2)}(\bu)V_N^{(0)}(\bu)
\prod_{n=1}^N u_n^{a-1} (1-u_n)^{b-1}\\=
\frac{2N^{s/2}S_N(a,b,1)}{s\Gamma_{N-1}(s)}
\frac{\Gamma\left(a+b-1+N\right)}
     {\Gamma\left(a-\frac{s}{2}+b-1+N\right)}
\frac{\Gamma\left(a-\frac{s}{2}\right)}
     {\Gamma(a)} 
\end{multline*}
and
\begin{multline*}
\int_0^\infty du_1\ldots\int_0^\infty du_n
V_N^{(s/2)}(\bu)V_N^{(0)}(\bu)
\prod_{n=1}^N u_n^{a-1} e^{-u_n/\theta}\\=
\prod_{n=0}^{N-1}\Gamma(a+n)\Gamma(2+n)
\frac{2\theta^{(a+N)N}}{s\Gamma_{N-1}(s)}
\left(\frac{N}{\theta} \right)^{s/2}
\frac{\Gamma\left(a-\frac{s}{2}\right)}
     {\Gamma(a)}.
\end{multline*}
\end{corollary}

\section{Conclusion}
It is well-known, even if it is not well understood, that there is a
connection between the random matrix theory and the Zeta function.
For example, Keating and its co-authors (\cite{keating2000random})
successfully use the characteristic polynomials $Z(U, \theta)$
of matrices $U$ in the Circular Unitary Ensemble (CUE)
to study the behavior of the correctly renormalized integral
$$
\int_0^T |\zeta(1/2+it)|^{2\lambda} dt
$$
Our work seems to be a first step in explaining this connection.
The Selberg integral plays a fundamental role in the theory of the various
$\beta$-ensembles (see \cite{forrester2010log}) and the Dixon-Anderson
probability distribution function is an intermediate step to the Selberg's
integral evaluation.

We hope that the results presented in this article will pave the way for
a deeper understanding of the links between these two fields.

\bibliographystyle{plain}
\bibliography{Alternating-Zeta}

\appendix

\section{Partial fraction decomposition}
\label{sec:pfd}

Let $Q$ be a polynomial of degree $N$ such that $ Q(x)=\prod_{n=1}^N (x-x_n)$ with
$x_i\neq x_j$ if $i\neq j$ and let $L_n(x)$, $n=1,\ldots N$ denote the Lagrange's polynomials
$$L_n(x) = \prod_{j=1,\, j\neq n}^N \frac{x-x_j}{x_n-x_j}.$$
If $P$ is a polynomial of degree strictly less than $N$ then by the Lagrange interpolation formula,
it can be written as
$$ P(x)=\sum_{n=1}^N P(x_n)L_n(x).$$
From this, we deduce the partial fraction expansion of $P/Q$
$$ \frac{P(x)}{Q(x)} =\sum_{n=1}^N P(x_n)\frac{L_n(x)}{Q(x)} = \sum_{n=1}^N  \frac{c_n}{x-x_n}
$$
with
$$
\quad c_n=\frac{P(x_n)}{\prod_{j=1, \, j\ne n}^N (x_n-x_j)} = \frac{P(x_n)}{Q'(x_n)}.
$$

\section{Inverse of the Vandermonde's matrix}
\label{app:inverse}
The proof of this formula can be found at 
\url{https://proofwiki.org/wiki/Inverse_of_Vandermonde_Matrix}
and is proposed as Exercise 40 from section 1.2.3 in \cite{knuth1968art}.
Consider the Vandermonde's matrix
\begin{equation}
\label{eq:VandermondeMatrix}
V_N(x_1,\ldots,x_N) = \begin{pmatrix}
  1    & x_1    & x_1^2  & \cdots & x_1^{N-1} \\
  1    & x_2    & x_2^2  & \cdots & x_2^{N-1} \\
\vdots & \vdots & \ddots & \vdots & \vdots \\
  1&   x_N      & x_N^2  & \cdots & x_N^{N-1} \\
\end{pmatrix}
\end{equation}

Assume $x_i\neq x_j$ if $i\neq j$, then the Vandermonde Determinant of $V_N$ is
$$\displaystyle \det \left({V_N(x_1,\ldots,x_N)}\right) = \prod_{1 \mathop \le i \mathop < j \mathop \le N} \left({x_j - x_i}\right) \ne 0.$$
Since this is non-zero, the inverse matrix, denoted $W_N = \left[{w_{ij}}\right]$, is guaranteed to exist.
Using the definition of the matrix product and the inverse matrix
$$\displaystyle \sum_{k \mathop = 1}^N x_i^{k-1} w_{kj}  = \delta_{ij}.$$
For $1\leq n \leq N$, if $P_n \left({x}\right)$ is the polynomial
$$\displaystyle P_n \left({x}\right) := \sum_{k \mathop = 1}^N w_{kn}x^{k-1}$$
then
$P_n \left({x_1}\right) = 0, \ldots, P_n \left({x_{n-1}}\right) = 0, P_n \left({x_n}\right) = 1
, P_n \left({x_{n+1}}\right) = 0, \ldots, P_n \left({x_N}\right) = 0.$
By the Lagrange's interpolation formula, the $n$th column of $W_N$ is composed of the
coefficients of the $n$th Lagrange basis polynomial

$$
\displaystyle P_n \left({x}\right) = \sum_{k \mathop = 1}^N w_{kn} x^{k-1}
 = \prod_{\substack {1 \mathop \le j \mathop \le N \\ j \mathop \ne n}} \frac {x - x_j} {x_n - x_j}.
$$
We can identify the terms $w_{ij}$ by expanding the product. In particular, setting $x=0$,
we get that the constant of the polynomials are
$$
w_{1n} =  \prod_{\substack {1 \mathop \le j \mathop \le N \\ j \mathop \ne n}} \frac {-x_j} {x_n - x_j}
= \prod_{\substack {1 \mathop \le j \mathop \le N \\ j \mathop \ne n}} \frac {x_j} {x_j - x_n},\quad n=1,\ldots,N.
$$

\section{An other determinant representation}

In this part we will show the following result:
\begin{proposition}
if $s\neq 0$, then
\begin{equation}\label{det:tridiag}
\eta_{N}(s) = \frac{1}{2}
\begin{vmatrix}
 1+\frac{1}{\lambda_{2,N}}           & -1                                    & 0      &  \ldots                                    & 0 \\
-\frac{\lambda_{2,N}}{\lambda_{3,N}} & 1+\frac{\lambda_{2,N}}{\lambda_{3,N}} &        &  \ddots                                    & 0  \\
 0                                   &                                       &        &                                            & \vdots \\
\vdots                               & \ddots                                & \ddots &  \ddots                                    &  0 \\
0                                    & \ddots                                &        &  1+\frac{\lambda_{N-2,N}}{\lambda_{N-1,N}} &-1 \\
0                                    & \ldots                                & 0      &  -\frac{\lambda_{N-1,N}}{\lambda_{N,N}}    & 1+\frac{\lambda_{N-1,N}}{\lambda_{N,N}}
\end{vmatrix}
\end{equation}
with $\lambda_{n,N}^{-1} = 2a_{n,N}(n^{-s}-1)$.
\end{proposition}
\begin{proof}
Let $W_N$ denote $V_N^{-1}$, then $|V^{(s)}_NW_N| = |V^{(s)}_N|\times|W_N| = |V^{(s)}_N|/|V_N|$. Some elementary
algebra shows that $V^{(s)}_N W_N$ is equal to the matrix
\begin{equation*}
\label{eq:matrix}V^{(s)}_N W_N=
\begin{pmatrix}
1                & 0                  & 0                  & \ldots & 0 \\
w_{11}(2^{-s}-1) & 1+w_{12}(2^{-s}-1) & w_{13}  (2^{-s}-1) & \ldots & w_{1N}  (2^{-s}-1) \\
w_{11}(3^{-s}-1) &   w_{12}(3^{-s}-1) & 1+w_{13}(3^{-s}-1) & \ldots & w_{1N}(3^{-s}-1)\\
\vdots           & \vdots             &                    &        &\vdots \\
w_{11}(N^{-s}-1) &   w_{12}(N^{-s}-1) & w_{13}(N^{-s}-1)   & \ldots & 1+w_{1N}(N^{-s}-1)
\end{pmatrix}.
\end{equation*}
Appendix \ref{app:inverse} reveals that $w_{1,n}=2a_{n,N}$ for $n=1,\ldots,N$.
If $s=0$, we get the identity matrix as expected, otherwise if $s\neq 0$, $\eta_{N}(s)$
can be written as
\begin{equation}\label{eq:onedet}
\eta_{N}(s) = \frac{1}{2}\prod_{n=2}^N 2a_{n,N} (n^{-s}-1)
\begin{vmatrix}
 1+ \frac{1}{2a_{2,N}(2^{-s}-1)} & 1                            & \ldots & 1 \\
 1                             & 1+\frac{1}{2a_{3,N}(3^{-s}-1)} & \ldots & 1 \\
\vdots                         & \ddots                         &        &   \vdots \\
1                              & 1                              & \ldots & 1+\frac{1}{2a_{N,N}(N^{-s}-1)} 
\end{vmatrix}.
\end{equation}
Setting $\lambda_{n,N}^{-1} =2a_{n,N}(n^{-s}-1)$ and using Gaussian
elimination method\footnote{From last rows/columns to bottom: ($c_{N-1}-c_{N-2}$, $l_{N-1}-l_{N-2}$,
$c_{N-2}-c_{N-3}$,..., $c_1-c_2$, $l_1-l_2$)} we get
\begin{equation*}
\eta_{N}(s) = \frac{1}{2}\prod_{n=2}^N \lambda_n^{-1}
\begin{vmatrix}
1+\lambda_{2,N} & -\lambda_{2,N}              & 0        & \ldots &\ldots        & 0 \\
 -\lambda_{2,N} & \lambda_{2,N}+\lambda_{3,N} & \ddots   &        &              & 0 \\
 0              & \ddots                      & \ddots   & \ddots & \ddots       & \vdots \\
\vdots          & \ddots                      & \ddots   & \ddots & \ddots              &  0 \\
0               &                             &          & \ddots &  \lambda_{N-1,N}+\lambda_{N-2,N} &-\lambda_{N-1,N} \\
0               & \dots                       &          & 0      & -\lambda_{N-1,N}                & \lambda_{N-1,N}+\lambda_{N,N}
\end{vmatrix}.
\end{equation*}
ending the proof.
\end{proof}

\begin{remark}
It is tempting to find a condition on the coefficients of the determinant (\ref{eq:onedet})
in order to find when it is zero.
Lemma \ref{cor:invertible} is equivalent to $\eta_{N}(s)=0$
if and only if $2\sum_{n=2}^N a_{n,N} (n^{-s}-1) =-1$. As $\sum_{n=1}^N a_{n,N} = 1/2$,
this equality can be rewritten
$\sum_{n=2}^N a_{n,N} n^{-s} = -a_{1,N}$ which is not really helpful ! 
\end{remark}

Assume $s\neq 0$, then, setting $\Delta_{0,N} = 0$, $\Delta_{1,N}=1$, the
determinant appearing in equation (\ref{det:tridiag}) can be computed using
the recurrence relation
\begin{equation}\label{eq:DeltaN}\tag{D}
\Delta_{n,N} = \left(1+\frac{\lambda_{n-1,N}}{\lambda_{n,N}}\right)\Delta_{n-1,N} - \frac{\lambda_{n-1,N}}{\lambda_{n,N}} \Delta_{n-2,N}, \quad n=2,\ldots,N
\end{equation}
with the convention $\lambda_{1,N} = 1$.
\begin{proposition}
Let
$$
\beta_{2,N}= \frac{1}{\lambda_{2,N}}\qquad \mbox{ and }\qquad
\beta_{n,N}= \frac{\lambda_{n-1,N}}{\lambda_{n,N}},  \qquad n=3,\ldots,N
$$
and let $\tilde{\Delta}_{n,N}$ denote the solution of the recurrence equation (\ref{eq:DeltaN})
with initial conditions $\tilde{\Delta}_{0,N} = 1$  and $\tilde{\Delta}_{1,N}=0$.
Then (see \cite{wall2018analytic}, p. 15 for the continued fraction
representation) for $n =2,\ldots, N$, we have
$$
\Delta_{n,N} = 1 + \sum_{k=0}^{n-2} \beta_{n-k}\ldots\beta_{2}
= 1+\sum_{k=2}^{n} \lambda_{k,N}^{-1}, \qquad n=2,\ldots, N,
$$
and
$$
\tilde{\Delta}_{n,N} = -\sum_{k=0}^{n-2} \beta_{n-k}\ldots\beta_{2}
= -\sum_{k=2}^{n} \lambda_{k,N}^{-1}, \qquad n=2,\ldots, N.
$$
Thus $\Delta_{N,N} = 2\eta_N(s)$, $\tilde{\Delta}_{N,N} = 1 - 2\eta_N(s)$
and
$$
\frac{1}{2\eta_N(s)} = 1 + \frac{1-\beta_{2,N}}{\beta_{2,N}+}
       \frac{1-\beta_{3,N}}{\beta_{3,N}+}\ldots\frac{1-\beta_{N-1,N}}{\beta_{N,N}}.
$$
\end{proposition}

\section{About the rank of a class of matrix}
\label{app:inversible}
 
\begin{lemma}\label{cor:invertible}
If $A$ is a $K\times K$ matrix with coefficients $a_{ij}=1$ if $i\neq j$ and $a_{ii} = 1+\lambda_i$
with $\lambda_i\neq0$ otherwise, then $|A|=0$ if and only if $\sum_{i=1}^K {\lambda_i}^{-1} = -1$.
\end{lemma}
\begin{proof}
$|A|=0$ if and only if there exists $K$ numbers $(\alpha_1,\ldots,\alpha_N)$ such that for some $i$,
$\alpha_i \neq 0$ and such that $\sum_{j=1}^N \alpha_j A^j = {0}$, where $A^j$ denotes the $j$th
column of $A$. Thus we have
$$
\alpha_i \lambda_i = - \sum_{j=1}^K \alpha_j,\quad i=1,\ldots,K.
$$
By assumption, all $\lambda_i$ are different from zero and for some $i\in\{1,\ldots,K\}$ there
exists $\alpha_i\neq 0$. This implies first that $\sum_{j=1}^K \alpha_j \neq 0$ and second that
in fact all $\alpha_i$ are different from 0. We get the identity
$$
\frac{-\alpha_i}{\sum_{j=1}^K \alpha_j} = \frac{1}{\lambda_i}, \quad i=1,\ldots,K.
$$
Summing we obtain the announced result.
\end{proof}
\begin{lemma}
Let $A$ is a $K\times K$ matrix with coefficients $a_{ij}=1$ if $i\neq j$ and $a_{ii} = 1+\lambda_i$
with $\lambda_i\neq0$ otherwise. If $S=1+\sum_{i=1}^K {\lambda_i}^{-1} \neq 0$
then $B=A^{-1}$ exists with coefficients $b_{ij}=-\frac{1}{S\lambda_i\lambda_j}$ if $i\neq j$
and $b_{ii}= \frac{1}{\lambda_i} - \frac{1}{S\lambda_i^2}$ otherwise.
\end{lemma}
\begin{proof}
Let $S_j = \sum_{i=1}^K b_{ij}$, if $B^j$ denotes the $j$th column of B,
identity $AB^j=e_j$ show that
\begin{equation*}
 \left\{ 
 \begin{array}{lcl}
  S_j + \lambda_i b_{ij} & = & 0 \qquad \mbox{ if } i\neq j \\
  S_j + \lambda_j b_{jj} & = & 1 \qquad \mbox{ otherwise. } \\
 \end{array}
 \right.
\end{equation*}
Thus $b_{ij} = -S_j/\lambda_i$ if $i\neq j$ and $b_{jj} = (1-S_j)/\lambda_j$. Summing and
equating, we find that $S_j = 1/(S\lambda_j)$ giving the announced result.
\end{proof}
\end{document}